\documentclass{article}
\usepackage{arxiv}
\usepackage{graphicx}
\usepackage{mathptmx}
\usepackage[subrefformat=parens,labelformat=parens]{subcaption}
\usepackage{authblk}
\usepackage{amsmath}
\usepackage{amsfonts}
\usepackage{amssymb}
\usepackage{amsbsy}
\usepackage{amsthm}
\usepackage{comment}
\usepackage{amsmath}

\usepackage[ruled, vlined]{algorithm2e}
\usepackage{bbm}
\usepackage{wrapfig}
\usepackage{booktabs}
\usepackage{multirow}
\usepackage{wrapfig,lipsum,booktabs}
\usepackage{hyperref}
\newtheorem{theorem}{Theorem}

\begin{document}

\title{\Large 
Adaptive 
Fast–Slow Operator Splitting for Multiscale Biochemical Stochastic Dynamics
}

\author[1]{Yuming Zeng}
\author[  ~,1]{Wei Xie\thanks{Corresponding author: w.xie@northeastern.edu}}
\author[1]{Keqi Wang}
\affil[1]{Northeastern University, Boston, MA 02115}

\maketitle

\section*{ABSTRACT}
Stochastic reaction networks governed by Chemical Langevin Equations (CLE) exhibit pronounced multiscale dynamics spanning fast molecular reactions, intermediate transport, and slow cellular regulation, posing significant challenges for efficient and accurate simulation. Although operator splitting naturally decouples fast and slow subsystems, a rigorous error characterization for CLE splitting schemes has been lacking. We propose a modular operator‑splitting framework with adaptive discretization that enables reliable and efficient simulation across fast–slow dynamics with explicit control of discretization error. Using stochastic logarithmic representations, we present a complete error analysis of the fast–slow Lie–Trotter splitting method, decomposing the one-step error into stochastic flow truncation error, commutator errors due to subsystem noncommutativity, and numerical discretization errors from 
fast and slow integrations. Guided by this analysis, we develop a proportional–integral (PI) adaptive controller that jointly selects macro time steps and fast microsteps, achieving substantial efficiency gains while maintaining accuracy.

\section{INTRODUCTION}

Biochemical reaction networks are inherently multiscale, comprising interacting subsystems that operate across distinct temporal and mechanistic regimes. Fast molecular processes such as enzymatic reactions and conformational changes evolve on microsecond to millisecond timescales. Intermediate processes, including metabolite transport and diffusion, occur over seconds. Slower cellular processes, such as metabolic flux redistribution, redox regulation, and biomass synthesis, unfold over minutes to hours. These subsystems are tightly coupled through nonlinear kinetics and regulatory interactions, forming a hierarchical and interdependent network subject to intrinsic stochasticity arising from thermodynamics, discrete reaction events, and molecular fluctuations \cite{rao2003stochastic,goutsias2013markovian}.

To model such systems, the Chemical Master Equation (CME) provides a rigorous stochastic description of reaction networks. However, for moderate to large molecular copy numbers, the CME becomes computationally intractable due to the combinatorial growth of its state space. The Chemical Langevin Equation (CLE) offers a widely used approximation by replacing discrete reaction events with continuous diffusion terms, yielding a stochastic differential equation (SDE) that captures intrinsic noise while remaining computationally tractable \cite{gillespie2000chemical,gillespie2007stochastic,kurtz1972relationship,wilkinson2018stochastic,higham2008modeling}. CLE-based models have been broadly applied in systems biology, metabolic engineering, and whole-cell modeling \cite{karr2012whole,browning2020identifiability}.

The multiscale structure of biochemical networks induces \textit{stiffness} in the CLE through widely separated reaction rates, leading to the coexistence of rapidly and slowly evolving modes. Both fast and slow reactions contribute to the CLE drift and diffusion terms; however, fast reaction channels generate rapidly varying components that dominate short‑time dynamics, while slow reactions primarily govern long‑term evolution.
This separation imposes severe timestep restrictions on numerical schemes: explicit methods must resolve the fastest stochastic fluctuations to ensure stability, resulting in prohibitively small stepsizes and limiting efficient long‑time simulation \cite{haseltine2002approximate,hepp2015adaptive}. 
Consequently, naive fixed stepsize Euler–Maruyama or Milstein schemes become computationally inefficient. 
Moreover, state-dependent diffusion, intrinsic stochasticity, and strongly
coupled nonlinear reaction rates further complicate the numerical behavior of
CLE systems and limit the applicability of classical stability analysis, which
typically relies on simplified assumptions such as additive noise, weak
nonlinearity, or decoupled dynamics \cite{gillespie1977exact,longtin2003effects,pucci2018deciphering}.



These challenges motivate the study of \emph{fast--slow decomposition} of biochemical reaction networks, in which rapid subsystems (e.g.\ glycolytic intermediates) are separated from slower processes, such as tricarboxylic acid (TCA) cycle dynamics \cite{wang2024multi} and pentose phosphate pathway (PPP) fluxes \cite{templeton2013peak,wang2024metabolic}, thereby enabling adaptive timestep selection.
Such decompositions arise naturally in singular perturbation theory \cite{pavliotis2008multiscale,berglund2006noise} and have long been employed heuristically in multiscale simulation. 
From a numerical perspective, this structure induces a splitting of the CLE generator,
\(
\mathcal{L} = \mathcal{L}_{\mathrm{fast}} + \mathcal{L}_{\mathrm{slow}},
\)
which can be evolved separately in time. When combined with operator splitting, the fast subsystem is integrated using multiple microsteps tailored to its dynamics, while the slow subsystem is advanced on a coarser grid, yielding efficient multiscale schemes whose accuracy depends on the interaction between the split operators.

However, extending operator splitting from deterministic stiff ODE/PDE solvers
to stochastic biochemical models introduces significant analytical challenges.
 In the Stratonovich formulation—natural for splitting—CLE dynamics are described in terms of stochastic flows and Kunita’s logarithmic expansion \cite{kunita1990stochastic,kunita2006stochastic}. 
Fast and slow stochastic vector fields generally do not commute, and the resulting commutator terms generate leading‑order splitting errors that do not arise in deterministic settings \cite{rossler2010runge}.  
Moreover, advancing the fast subsystem with $N$ microsteps  per macro‑step introduces discretization errors that interact nontrivially with diffusive fluctuations, thereby altering the effective order of accuracy.
Classical strong and weak SDE error analyses, which focus on local truncation error for a single discretization and assume a fixed generator, cannot fully capture these effects. 
In particular, the Baker–Campbell–Hausdorff structure couples state‑dependent drift, multiplicative diffusion, and substep discretization in a manner specific to reaction‑network CLEs
\cite{castell1993asymptotic,malham2008stochastic}.

Despite the prevalence of multiscale biochemical SDE models and the widespread use of operator splitting, \emph{a unified mean-square error (MSE) characterization for fast–slow splitting of the CLE is still lacking}.  
Existing studies focus primarily on deterministic systems, weak‑order expansions, or simplified additive‑noise settings \cite{kulchitski2000unified,Misawa2000,strang1968construction,liu2022convergence,kloeden1977numerical}, and do not extend to fully multiplicative reaction‑network models in which both drift and diffusion depend nonlinearly on state variables such as species concentrations.


To fill this gap, we present a rigorous,
self-contained MSE analysis for fast-slow splitting of CLE and develop an adaptive discretization strategy that enables reliable and efficient simulation across disparate time scales with explicit error control. 
Starting from a modular Stratonovich SDE formulation, we decompose the dynamics into fast and slow subsystems and integrate them using operator splitting. 
This multiscale discretization introduces multiple sources of numerical error, including stochastic flow truncation, fast–slow commutator effects, and discretization errors in the fast and slow subflows.
We quantify these contributions through a unified MSE decomposition that depends explicitly on the macro‑step size $\Delta t$ and the number of fast substeps $N$. This error structure then guides the adaptive selection
of $(\Delta t, N)$, enabling efficient allocation of computational effort across fast and slow dynamics.

Our main contributions are threefold.
\textbf{(1) Stochastic-flow formulation of fast--slow CLE splitting.}  
Using Kunita’s stochastic logarithm, we represent both the exact CLE flow and
its fast--slow splitting approximation as exponentials of random vector fields, enabling a systematic Lie‑bracket expansion and explicit characterization of stochastic commutator errors \cite{kunita1990stochastic,Misawa2000}.
\textbf{(2) One-step MSE characterization.}  
Given the current state $\pmb{s}$, we derive a complete local MSE expansion of the form
\(
E(\pmb{s};\Delta t,N)
= \Delta t^{2}\!\left(A(\pmb{s})+\frac{B(\pmb{s})}{N}\right)
+ o(\Delta t^{2}),
\)
where \(A(\pmb{s})\) captures truncation, commutator, and slow discretization errors,
while \(B(\pmb{s})/N\) quantifies discretization errors from fast substepping.
\textbf{(3) Adaptive selection of \(\Delta t\) and \(N\).}  
Guided by this error decomposition, we develop an adaptive strategy that jointly
selects the macro-step size \(\Delta t\) and the number of fast substeps \(N\) to achieve prescribed accuracy at minimal computational cost. We implement this strategy using a proportional–integral (PI) controller \cite{aastrom2013computer,ilie2015adaptive},
though the framework readily extends to other adaptive schemes.

The remainder of the paper is organized as follows. 
Section~\ref{sec:problemDescription} introduces the fast--slow stochastic reaction network model and its
Stratonovich vector-field representation. 
Section~\ref{sec:error} presents the mean--square error analysis of the fast--slow splitting
integrator, including truncation, commutator, and discretization errors. 
Section~\ref{sec:MDE-Driven adpative control} develops the MSE-driven adaptive strategy for selecting the macro-step
size and fast substepping resolution. 
Section~\ref{sec:Empirical Study} evaluates the finite sample performance of the proposed method through numerical experiments on stiff biochemical reaction networks, and Section~\ref{sec:conclusion} concludes the paper.

\section{FAST--SLOW STOCHASTIC REACTION NETWORK MODEL}
\label{sec:problemDescription}

To capture stochastic fluctuations across disparate time scales, we model biochemical reaction networks using Stratonovich SDEs, with both drift and diffusion decomposed into fast and slow components.
Consider a stochastic reaction network with $N_s$ species, whose concentrations at time $t$ are denoted by
$\pmb{s}_t\in\mathbb{R}^{N_s}$. The network consists of $N_r$ reaction channels.
For each $k$-th reaction, let $C_{:,k}\in\mathbb{R}^{N_s}$ denote the stoichiometric
change vector and $v_k(\pmb{s})\ge 0$ the propensity function. Collectively, we define the stoichiometric matrix
$
C = \big[\, C_{:,1}, C_{:,2}, \ldots, C_{:,N_r}\,\big]$ and reaction rates $ 
\pmb{v}(\pmb{s})=\big[v_1(\pmb{s}), \ldots, v_{N_r}(\pmb{s})\big]^\top.$ Let $d\pmb{R}_t\in\mathbb{R}^{N_r}$ denote the vector of reaction counts over
an infinitesimal interval $(t,t+dt]$, where $dR_{t,k}$ is the number of
occurrences of reaction $k$. Conditioned on the current state $\pmb{s}_t$,
we assume that reactions occur independently and do not coincide in time, so that
$d\pmb{R}_t$ are independent. Consequently, $d\pmb{R}_t$ follows a multivariate Poisson process with
mean $\mathbb{E}[d\pmb{R}_t]=\pmb{v}(\pmb{s}_t)\,dt$ and covariance
$\mathrm{Cov}(d\pmb{R}_t)=
\mathrm{diag}\!\big(v_1(\pmb{s}_t),\ldots,v_{N_r}(\pmb{s}_t)\big)\,dt.$
Using the diffusion approximation \cite{gillespie2000chemical}, we approximate $d\pmb{R}_t$ by
\(
d\pmb{R}_t
=
\pmb{v}(\pmb{s}_t)\,dt
+
\mathrm{diag}\!\big(\pmb{v}(\pmb{s}_t)\big)^{1/2} d\pmb{W}_t,
\)
where $\pmb{W}_t\in\mathbb{R}^{N_r}$ is a standard Wiener process. The state evolution then satisfies
\(
d\pmb{s}_t = C\,d\pmb{R}_t
\), 
yielding the Chemical Langevin Equation
\[
d\pmb{s}_t
=
C\,\pmb{v}(\pmb{s}_t)\,dt
+
\{C\,\mathrm{diag}\!\big(\pmb{v}(\pmb{s}_t)\big)C^T\}^{1/2} d\pmb{W}_t.
\]

Many biochemical networks exhibit intrinsic multiscale dynamics. We therefore partition the reaction channels into \emph{slow (s)} and \emph{fast (f)} subsets: $\mathcal{I}_{\mathrm{slow}}\cup\mathcal{I}_{\mathrm{fast}}=\{1,2,\ldots,N_r\}$
and $\mathcal{I}_{\mathrm{slow}}\cap\mathcal{I}_{\mathrm{fast}}=\emptyset$.
Reordering indices yields
$
C=[\,C_{\mathrm{slow}}\;\;C_{\mathrm{fast}}\,]$ and $
\pmb{v}(\pmb{s})
=
[
\pmb{v}_{\mathrm{slow}}(\pmb{s})^T,
\pmb{v}_{\mathrm{fast}}(\pmb{s})^T
]^T.$
Under this decomposition, the drift splits naturally as
\(
f(\pmb{s}) = C\,\pmb{v}(\pmb{s})
= f_{\mathrm{slow}}(\pmb{s}) + f_{\mathrm{fast}}(\pmb{s}),
\)
with
\(
f_{\mathrm{slow}}(\pmb{s}) = C_{\mathrm{slow}}\pmb{v}_{\mathrm{slow}}(\pmb{s}),\;
f_{\mathrm{fast}}(\pmb{s}) = C_{\mathrm{fast}}\pmb{v}_{\mathrm{fast}}(\pmb{s}).
\)
Similarly, defining the diffusion covariance
\(
\Sigma(\pmb{s}) = C\,\mathrm{diag}(\pmb{v}(\pmb{s}))\,C^\top,
\)
we obtain the corresponding fast–slow decomposition
$
\Sigma_{\mathrm{slow}}(\pmb{s})
= C_{\mathrm{slow}}\mathrm{diag}(\pmb{v}_{\mathrm{slow}}(\pmb{s}))C_{\mathrm{slow}}^\top$ and 
$\Sigma_{\mathrm{fast}}(\pmb{s})
= C_{\mathrm{fast}}\mathrm{diag}(\pmb{v}_{\mathrm{fast}}(\pmb{s}))C_{\mathrm{fast}}^\top.
$

A Stratonovich representation is particularly convenient for stochastic-flow
analysis and operator splitting. Starting from the CLE, the Stratonovich drift takes the form
$
f^\circ(\pmb{s})
=
C\,\pmb{v}(\pmb{s})
-\tfrac14\,C\,d(\pmb{s}),
$
where the correction vector $d(\pmb{s})\in\mathbb{R}^{N_r}$ has components
$
d_k(\pmb{s})=C_{:,k}^\top \nabla v_k(\pmb{s})$ for $
k=1,2,\ldots,N_r.$
To facilitate stochastic‑flow analysis, we introduce the associated vector‑field representation. The drift vector field is defined as
\(
X_0(\pmb{s})=f^\circ(\pmb{s}),
\)
while each $k$-th reaction channel defines a diffusion vector field
$
X_k(\pmb{s})=C_{:,k}\sqrt{v_k(\pmb{s})}$. Under the fast–slow partition, the drift decomposes as
$X_0(\pmb{s})=X_0^{(s)}(\pmb{s})+X_0^{(f)}(\pmb{s})$, where
\[
X_0^{(s)}(\pmb{s})
=
C_{\mathrm{slow}}\pmb{v}_{\mathrm{slow}}(\pmb{s})
-\tfrac14 C_{\mathrm{slow}}d_{\mathrm{slow}}(\pmb{s}),
\qquad
X_0^{(f)}(\pmb{s})
=
C_{\mathrm{fast}}\pmb{v}_{\mathrm{fast}}(\pmb{s})
-\tfrac14 C_{\mathrm{fast}}d_{\mathrm{fast}}(\pmb{s}).
\]
Using this decomposition, the Stratonovich CLE can be written in standard Kunita form as,
\begin{equation}
\label{eq:Kunita_form}
d\pmb{s}_t
=
X_0(\pmb{s}_t)\,dt
+
\sum_{k\in\mathcal{I}_{\mathrm{slow}}}
X_k(\pmb{s}_t)\circ dW_{t,k}
+
\sum_{j\in\mathcal{I}_{\mathrm{fast}}}
X_j(\pmb{s}_t)\circ dW_{t,j}.
\end{equation}


\section{MEAN‑SQUARE ERROR ANALYSIS OF FAST–SLOW OPERATOR SPLITTING}
\label{sec:error}

In this section, we analyze the local mean‑square error (MSE) of the fast–slow operator‑splitting method for SDE‑based mechanistic bioprocess models. Using the Stratonovich vector‑field formulation introduced in Section~\ref{sec:problemDescription} together with Kunita’s stochastic flow expansion, we decompose the one‑step error into four components: (i) truncation of the Kunita stochastic flow, (ii) fast–slow operator noncommutativity, (iii) discretization error from fast substepping, and (iv) discretization error of the slow subflow.


\subsection{KUNITA TRUNCATION AND FAST–SLOW OPERATOR‑SPLITTING ERRORS FOR SDES}


We first derive two sources of error arising from stochastic flow truncation and SDE operator splitting.

\textbf{(i) Kunita Truncation Error Quantification.}
Consider the Stratonovich SDE \eqref{eq:Kunita_form}, and let $\Phi_{\Delta t}$
denote its stochastic flow map over the interval $[t, t+\Delta t]$. Following Kunita's
representation \cite{kunita1990stochastic}, the flow can be written in exponential form
\(
\Phi_{\Delta t} = \exp(Y_{\Delta t}),
\)
where $Y_{\Delta t}$ is  the logarithm of the flow and possesses a stochastic Lie‑series expansion, i.e.,
\begin{equation}
Y_{\Delta t}
= J_{(0)}\, X_0
+ \sum_{k\in\mathcal{I}_{\mathrm{slow}}} J_{(k)}\, X_k
+ \sum_{j\in\mathcal{I}_{\mathrm{fast}}} J_{(j)}\, X_j
+ \sum_{i<j} J_{(i,j)}\, [X_i, X_j]
+ \cdots ,
\label{eq.KunitaY}
\end{equation}
where the coefficients $J_{(\alpha)}$ are iterated Stratonovich integrals
over $[t, t+\Delta t]$. 
Grouping terms in (\ref{eq.KunitaY}) by order yields
\(
Y_{\Delta t} = Y^{(1)} + Y^{(2)} + Y^{(3)} + \cdots,
\)
where $Y^{(1)}$ contains first-order increments and $Y^{(2)}$ comprises second‑order contributions involving iterated integrals and Lie brackets. Truncation at first order 
gives
\begin{equation}
\label{formula:truncate}
    \hat{Y}_{\Delta t} = Y^{(1)}
= \Delta t\,X_0
+ \sum_{k\in\mathcal{I}_{\mathrm{slow}}} \Delta W_k\,X_k
+ \sum_{j\in\mathcal{I}_{\mathrm{fast}}} \Delta W_j\,X_j,
\end{equation}
with the corresponding truncated exponential flow map represented as $\exp(\hat{Y}_{\Delta t})$.


\begin{theorem}[Kunita Truncation Error]
\label{lem:truncation}
Let
$
S_{n+1} = \exp(Y_{\Delta t})(S_n)$ and $
\widehat{S}_{n+1} = \exp(\widehat{Y}_{\Delta t})(S_n),$
where $Y_{\Delta t}$ is the exact Stratonovich logarithm of the stochastic flow
over a time step $\Delta t$ starting from state $S_n=\pmb{s}$, and $\widehat{Y}_{\Delta t}$ denotes its first-order
truncation defined in Equation~\eqref{formula:truncate}.
Then the local mean-square error satisfies
\begin{equation}
\label{eq:Kunita_truncation}
\mathbb{E}\!\left[
\|S_{n+1}-\widehat{S}_{n+1}\|^2
\mid S_n=\pmb{s}
\right]
=
\frac14\,\Delta t^2
\sum_{i<j}
\|[X_i,X_j](\pmb{s})\|^2
+
O(\Delta t^3),
\end{equation}
where $[X_i,X_j]$ denotes the Lie bracket of the vector fields.
\end{theorem}

\begin{proof}[Proof sketch]
Write $Y_{\Delta t}=Y^{(1)}+Y^{(2)}+Y^{(3)}+\ldots$ and
$\hat Y_{\Delta t}=Y^{(1)}$.
A Fréchet expansion of the exponential map yields
$S_{n+1}-\hat S_{n+1}
=\exp(Y_{\Delta t})(\pmb{s})-\exp(\hat Y_{\Delta t})(\pmb{s})
=Y^{(2)}(\pmb{s})+\mathcal R_{\mathrm{trunc}}(\pmb{s})$, where
$\mathbb E\|\mathcal R_{\mathrm{trunc}}\|^{2}=o(\Delta t^3)$.
Consequently,
$\mathbb E[\|S_{n+1}-\hat S_{n+1}\|^{2}\mid S_n=\pmb{s}]
=\mathbb E[\|Y^{(2)}(\pmb{s})\|^{2}\mid S_n=\pmb{s}]+o(\Delta t^{3})$.
In Itô form, the second‑order term can be written as
$Y^{(2)}(\pmb{s})
=\sum_{1\le i<j\le N_r}H^{(i,j)}[X_i,X_j](\pmb{s})
 +\sum_{i=1}^{N_r} H^{(0,i)}[X_0,X_i](\pmb{s})$,
where 
$H^{(i,j)}=\tfrac12(I_{(i,j)}-I_{(j,i)})$ and $H^{(0,i)}=I_{(0,i)}$ are
centered Itô iterated integrals with
$I_{(i,j)} = \int_t^{t+\Delta t} \int_t^{\tau} dW_{s,i}\, dW_{\tau,j}$ and $
I_{(0,i)} = \int_t^{t+\Delta t} \int_t^{\tau} ds\, dW_{\tau,i}.$ 
Mixed covariances contribute only at higher
order, so only diagonal terms remain.
Using $\mathbb E[(H^{(i,j)})^{2}]=\tfrac14(\Delta t)^2$ and
$\mathbb E[(H^{(0,i)})^{2}]=\tfrac13(\Delta t)^3$ completes the proof.
\end{proof}

\textbf{(ii) Fast-Slow Operator Splitting Error.}
We next decompose the logarithmic flow $Y_{\Delta t}$ into fast and slow
reaction components. Over a macro‑step $\Delta t$, the Lie–Trotter splitting
integrator advances the solution as
$S^{\mathrm{split}}_{n+1}
= \exp\!\big(Y_{\mathrm{slow}}^{\Delta t,\Delta W^{(s)}}\big)\;
  \exp\!\big(Y_{\mathrm{fast}}^{\Delta t,\Delta W^{(f)}}\big)\; S_n$,
that is, the slow flow is applied after the fast flow. 

\begin{theorem}[Fast--Slow Noncommutativity Error]
\label{lem:splitting}
Let $
Y_{\mathrm{fast}}^{\Delta t,\Delta W^{(f)}}
=
\Delta t\,X_0^{(f)}
+
\sum_{i\in\mathcal{I}_{\mathrm{fast}}}
\Delta W_i\,X_i$ and $
Y_{\mathrm{slow}}^{\Delta t,\Delta W^{(s)}}
=
\Delta t\,X_0^{(s)}
+
\sum_{j\in\mathcal{I}_{\mathrm{slow}}}
\Delta W_j\,X_j$.
By applying the Lie--Trotter approximation
$S^{\mathrm{split}}_{n+1}=\exp(Y_{\mathrm{slow}}^{\Delta t,\Delta W^{(s)}})\exp(Y_{\mathrm{fast}}^{\Delta t,\Delta W^{(f)}})S_n$,
the local mean‑square fast–slow splitting error satisfies
\begin{equation}
\label{eq:splitting_error}
\mathbb{E}\!\left[
\|S_{n+1}^{\mathrm{split}}-\widehat{S}_{n+1}\|^2
\mid S_n=\pmb{s}
\right]
=
\frac14\,\Delta t^2
\sum_{i\in\mathcal{I}_{\mathrm{fast}}}
\sum_{j\in\mathcal{I}_{\mathrm{slow}}}
\|[X_i,X_j](\pmb{s})\|^2
+
o(\Delta t^2).
\end{equation}
\end{theorem}

\begin{proof}
By the general Baker–Campbell–Hausdorff (BCH) formula,
$\exp(A)\exp(B)=\exp(A+B+\tfrac12[A,B]+\ldots)$.
Since $\|Y_{\mathrm{fast}}^{\Delta t,\Delta W^{(f)}}\|,\|Y_{\mathrm{slow}}^{\Delta t,\Delta W^{(s)}}\|=O(\Delta t^{1/2})$ ,
all higher‑order BCH terms contribute only $O(\Delta t^3)$ in mean square.
The leading contribution therefore comes from the commutator term
$[A_{\Delta t},B_{\Delta t}]
=\sum_{i\in\mathcal I_{\mathrm{fast}}}\sum_{j\in\mathcal I_{\mathrm{slow}}}
\Delta W_i\Delta W_j[X_i,X_j]$.
Because Wiener increments are independent,
$\mathbb E[(\Delta W_i\Delta W_j)^2]=\Delta t^2$ for $i\neq j$,
yielding the stated $O(\Delta t^2)$ mean‑square error.
\end{proof}



\subsection{FAST SUBSTEPPING AND SLOW DISCRETIZATION ERRORS}
\label{subsec:substepping_order}


The exponential map admits a natural flow interpretation: for a given (random) vector
field $Y_{\mathrm{fast}}^{\Delta t,\Delta W^{(f)}}$, the mapping
$\exp\!\big(Y_{\mathrm{fast}}^{\Delta t,\Delta W^{(f)}}\big)$ corresponds to
the time-$1$ flow of the deterministic ordinary differential equation (ODE)
generated by this vector field \cite{Misawa2000}. Specifically, 
$
\exp\!\big(Y_{\mathrm{fast}}^{\Delta t,\Delta W^{(f)}}\big)(\pmb{s})
= \phi_{\mathrm{fast}}(1;\pmb{s}),
$
where $\phi_{\mathrm{fast}}(\tau;\pmb{s})$ solves
$
\frac{d\phi_{\mathrm{fast}}(\tau;\pmb{s})}{d\tau}
= Y_{\mathrm{fast}}^{\Delta t,\Delta W^{(f)}}\!\big(\phi_{\mathrm{fast}}(\tau;\pmb{s})\big)$
 and $\phi_{\mathrm{fast}}(0;\pmb{s})=\pmb{s} 
$ for $ \tau \in [0,1]$.
Because the vector field depends nonlinearly on the state through the reaction propensities, a closed‑form solution is generally unavailable and the flow must be approximated numerically. 
Exploiting the smallness of $\Delta t$, a common and efficient approximation is to evaluate the vector field at the initial state of each substep, yielding simple explicit updates that are inexpensive to compute. This approximation is used below to quantify both fast‑substepping and slow‑discretization errors.


\vspace{0.1in}
\textbf{(iii) Discretization Error from Fast Substepping.}
For the fast subsystem, we further subdivide the macro-step into $N$ uniform micro‑steps of size $\delta t=\Delta t/N$.  At the operator level,
\begin{equation}
\label{eq:fast_subdivide_exact}
\exp\!\big(Y_{\mathrm{fast}}^{\Delta t,\Delta W^{(f)}}\big)
=
\Big(
\exp\!\big(Y_{\mathrm{fast}}^{\Delta t/N,\Delta W^{(f)}/N}\big)
\Big)^{N},
\end{equation}
since $Y_{\mathrm{fast}}^{\Delta t,\Delta W^{(f)}}$ is the sum
of $N$ identical increments and the exponential of a sum of $N$ identical operators equals the $N$-fold composition of the exponential of the single increment.
Thus, the algebraic subdivision of the exact
fast exponential into $N$ identical exponentials introduces no additional operator‑splitting error; it is exact at the continuous‑operator level.

Each factor $\exp\!\big(Y_{\mathrm{fast}}^{\delta t,\Delta W^{(f)}/N}\big)$ is still approximated numerically due to the nonlinearity of its generating ODE. A standard approximation freezes the vector field at the substep initial state.
For example, for the first substep, given the current state $S_n = \pmb{s}$, one may approximate
\[
\frac{d\phi(\tau)}{d\tau}
=
Y_{\mathrm{fast}}^{\delta t,\Delta W^{(f)}/N}\!\big(\phi(\tau)\big)
\quad\text{by}\quad
\frac{d\tilde\phi(\tau)}{d\tau}
=
Y_{\mathrm{fast}}^{\delta t,\Delta W^{(f)}/N}\!\big(\pmb{s}\big),
\]
whose solution is
\(
\tilde\phi(\tau)
=
\pmb{s}
+
\tau\,Y_{\mathrm{fast}}^{\delta t,\Delta W^{(f)}/N}(\pmb{s}).
\)
This frozen‑field approximation is equivalent to an explicit Euler‑type update at each micro‑step. 
The resulting ODE substepping error arises purely from numerical discretization and is distinct from the operator‑splitting error, which originates from noncommutativity between fast and slow operators.

\begin{sloppypar}
\begin{theorem}[Accumulated Fast-Substepping MSE with Propagation]
\label{lem:fast_accum_propag}
Assume each diffusion operator \(X_i\) is a
first--order differential operator: 
$
X_i=\sum_{\alpha=1}^{N_s} X_{i,\alpha}(\pmb{s})\,\partial_{s_\alpha}$ for $i\in \mathcal I_{\mathrm{fast}}.$
Then
\[
\mathbb E\!\left[\|E_m(\pmb{s})\|^2\right]
=
\frac14\,\delta t^2
\sum_{i,j\in \mathcal{I}_{\mathrm{fast}}} c_{ij}
\sum_{\alpha=1}^{N_s}
\Bigg(
   \sum_{\beta=1}^{N_s}
      \partial_{s_\beta}X_{i,\alpha}(\pmb{s})\,X_{j,\beta}(\pmb{s})
\Bigg)^{\!2}
+ o(\delta t^2)
\]
holds uniformly for $\pmb{s}$, where $c_{ij}=3$ if $i=j$ and $c_{ij}=1$ if $i\neq j$. For an integer $N\ge 1$, we define the microstep operator
$
\Phi_{\delta t}(u)
:= \exp\!\left( 
     Y_{\mathrm{fast}}^{\delta t,\,\Delta W^{(f)}/N}
   \right)(u),$ with $
\delta t=\Delta t/N,$
and the sequence of fast microstep states
$
\pmb{s}^{(0)} := \pmb{s}$ and $
\pmb{s}^{(m)} := \Phi_{\delta t}\big(\pmb{s}^{(m-1)}\big)$ with $m=1,2,\dots,N.$
Let $E_m(\pmb{s}^{(m-1)})$ denote the local discretization error generated at
microstep $m$ when
the microstep starts from the (random) state $\pmb{s}^{(m-1)}$. Assume that $X_0^{(f)}$ and $X_i$ are locally bounded in a
neighborhood of $\pmb{s}$. 
Then the accumulated mean square error over $N$ microsteps satisfies
\[
\mathbb E\!\left[\Big\|\sum_{m=1}^N E_m\big(\pmb{s}^{(m-1)}\big)\Big\|^2\right]
=
\frac14\,\Delta t^2\,\frac{1}{N}
\sum_{i,j\in \mathcal{I}_{\mathrm{fast}}} c_{ij}
\sum_{\alpha=1}^{N_s}
\Bigg(
   \sum_{\beta=1}^{N_s}
      \partial_{s_\beta}X_{i,\alpha}(\pmb{s})\,X_{j,\beta}(\pmb{s})
\Bigg)^{\!2}
+ R(\pmb{s};\Delta t,N),
\]
where the remainder satisfies
\(
R(\pmb{s};\Delta t,N)
= o\!\Big(\frac{\Delta t^2}{N}\Big)
\) as \(\Delta t\to 0,\ \frac{\Delta t}{N}\to 0.
\)
Thus, the leading-order term is the same as obtained by evaluating each local
variance at the initial state $\pmb{s}$, while propagation and cross-step correlations
only contribute higher-order corrections.
\end{theorem}
\end{sloppypar}

\begin{proof}
For the $m$-th fast microstep, the frozen first--order Kunita vector field is
$A^{(m)}=\delta t\,X_0^{(f)}+\sum_{i\in\mathcal I_{\mathrm{fast}}}\Delta W^{i,(m)}X_i$,
with $\Delta W^{i,(m)}\sim\mathcal N(0,\delta t)$ independent across $m$.
Let $\phi^{(m)}(1;\cdot)=\exp(A^{(m)})$ and denote the Euler defect
$E_m(u)=\phi^{(m)}(1;u)-(u+A^{(m)}(u))$. By the Taylor expansion of the flow map,
$E_m(u)=\tfrac12(A^{(m)})'(u)A^{(m)}(u)+O(\|A^{(m)}(u)\|^3)$.
Since $\|A^{(m)}\|=O(\delta t^{1/2})$, only the diffusion--diffusion part of
$(A^{(m)})'A^{(m)}$ contributes at order $\delta t^2$ in mean square.
Standard Gaussian moment
calculations give
$\mathbb E\|E_m(\pmb{s})\|^2=\tfrac14\,\delta t^2 S(\pmb{s})+o(\delta t^2)$,
where $S(\pmb{s})=\sum_{i,j\in \mathcal I_{\mathrm{fast}}} c_{ij}
\sum_{\alpha=1}^{N_s}
\left(
   \sum_{\beta=1}^{N_s}
      \partial_{s_\beta}X_{i,\alpha}(s)\,X_{j,\beta}(s)
\right)^{\!2}$. Consider the accumulated error
$\mathcal E=\mathbb E\|\sum_{m=1}^N E_m(\pmb{s}^{(m-1)})\|^2$.
By local boundedness of the fast flow,
$\mathbb E\|\pmb{s}^{(m)}-\pmb{s}\|^2=O(\Delta t)$ uniformly in $m$.
Hence
$\mathbb E\|E_m(\pmb{s}^{(m-1)})\|^2
=\mathbb E\|E_m(\pmb{s})\|^2+o(\delta t^2)$,
and summing the diagonal terms yields
$\sum_{m=1}^N\mathbb E\|E_m(\pmb{s}^{(m-1)})\|^2
=\tfrac14\,\Delta t^2 N^{-1} S(\pmb{s})+o(\Delta t^2/N)$. For $m\neq k$, independence of Wiener increments implies
$\sum_{m<k}\mathbb E[E_m^\top E_k]=o(\Delta t^2/N)$.
Combining both contributions completes the proof.
\end{proof}

\textbf{(iv) Discretization Error of Slow Subflow.}
To complete the one--step error analysis, we account for the discretization error introduced by the slow subsystem.
Unlike the fast dynamics, which are integrated using multiple microsteps, the slow subsystem is advanced by a single Euler update over the macro‑step
$\Delta t$. By an argument analogous to the proof of Theorem~\ref{lem:fast_accum_propag}, we obtain the following result.

\begin{theorem}[Euler Discretization Error of Slow Flow]
\label{lem:slow_ODE_simple}
Let $s^{+}$ denote the state obtained after evolving the fast subflow over a macro‑step $\Delta t$.
Consider the exact slow flow 
$\exp(Y_{\mathrm{slow}}^{\Delta t,\Delta W^{(s)}})(\pmb{s}^{+})$
and its Euler approximation 
$\pmb{s}^{+} + \Delta t\,X_0^{(s)}(\pmb{s}^{+})$.
Then, the local one‑step mean‑square error satisfies
\begin{equation}
\label{eq:slow_ODE_error_simple}
\mathbb E\!\big[\|E_{\mathrm{slow}}(\pmb{s}^{+})\|^2\big]
=
\frac14\,\Delta t^2
\sum_{j,k\in \mathcal{I}_{\mathrm{slow}}} c_{jk}
\sum_{\alpha=1}^{N_s}
\Big(
   \sum_{\beta=1}^{N_s}
      \partial_{s_\beta}X_{j,\alpha}(\pmb{s})\,X_{k,\beta}(\pmb{s})
\Big)^{\!2}
+ o(\Delta t^2).
\end{equation}
\end{theorem}


\subsection{TOTAL ONE‑STEP MEAN‑SQUARE ERROR ANALYSIS}

Since all leading--order contributions established in 
Theorems~\ref{lem:truncation}--\ref{lem:slow_ODE_simple}
scale quadratically with $\Delta t$, we can combine them to obtain the total one--step mean square error, as stated in Theorem~\ref{thm:total_MSE}. In particular, the one‑step MSE admits the asymptotic expansion
\(
E(\pmb{s};\Delta t,N)
=
\Delta t^2
\bigl(
A(\pmb{s})+\tfrac{B(\pmb{s})}{N}
\bigr)
+ o(\Delta t^2),
\)
where $A(\pmb{s})$ collects all contributions independent of the miscrostepping parameter $N$, including stochastic‑flow truncation, fast–slow splitting, and slow‑subflow discretization, while $B(\pmb{s})$ captures the error from fast substepping.

\begin{theorem}[One-step MSE Decomposition]
\label{thm:total_MSE}
Combining 
Theorems~\ref{lem:truncation}--\ref{lem:slow_ODE_simple},
the one-step mean square error of the fast--slow operator splitting integrator satisfies
\begin{equation}
\label{eq:total_MSE}
\begin{aligned}
& E(\pmb{s};\Delta t,N)
= \underbrace{\frac{1}{4}\,\Delta t^{2}
    \sum_{1\le i<j\le N_r}\big\|[X_i,X_j](\pmb{s})\big\|^2}_{\text{(i) Stochastic flow truncation error}}
  \;+\; \underbrace{\frac{1}{4}\,\Delta t^{2}
    \sum_{i\in \mathcal{I}_{\mathrm{fast}}}\sum_{j\in \mathcal{I}_{\mathrm{slow}}}\,
        \big\|[X_i,X_j](\pmb{s})\big\|^2}_{\text{(ii) Fast-slow operator splitting error}}
\\[6pt]
& +\underbrace{\frac14\,\frac{\Delta t^2}{N}
    \sum_{i,j\in \mathcal{I}_{\mathrm{fast}}} c_{ij}\sum_{\alpha=1}^{N_s}
    \Big(
       \sum_{\beta=1}^{N_s}
          \partial_{s_\beta}X_{i,\alpha}(\pmb{s})\,X_{j,\beta}(\pmb{s})
    \Big)^{\!2}}_{\text{(iii) Accumulated discretization error of fast substepping}}+\underbrace{\frac14\,\Delta t^2
    \sum_{j,k\in \mathcal{I}_{\mathrm{slow}}} c_{jk}\sum_{\alpha=1}^{N_s}
    \Big(
       \sum_{\beta=1}^{N_s}
          \partial_{s_\beta}X_{j,\alpha}(\pmb{s})\,X_{k,\beta}(\pmb{s})
    \Big)^{\!2}}_{\text{(iv) Discretization error of slow subflows}}+ o(\Delta t^2).
\end{aligned}
\end{equation}
\begin{proof}[Proof sketch]
Decompose the one‑step error as
\(
e = e_{\mathrm{trunc}} + e_{\mathrm{split}} + e_{\mathrm{fast}} + e_{\mathrm{slow}},
\)
corresponding to stochastic flow truncation, fast–slow splitting, fast substepping, and slow discretization errors, respectively; as presented in
Theorems~\ref{lem:truncation}--\ref{lem:slow_ODE_simple}.
Then
\(
\mathbb{E}\|e\|^2
= \sum_\ell \mathbb{E}\|e_\ell\|^2
+ \sum_{\ell\neq m} \mathbb{E}\langle e_\ell, e_m\rangle.
\) Each error component contributes a leading‑order term of order $O(\Delta t^2)$, with the fast substepping error scaled by $1/N$. Owing to the independence and orthogonality of the underlying stochastic increments, all cross term $\mathbb{E}\langle e_\ell, e_m\rangle$ are of higher order and do
not contribute to the leading term $O(\Delta t^2)$. Collecting the leading contributions yields \eqref{eq:total_MSE}.
\end{proof}

\end{theorem}




Parametrizing the drift and diffusion vector fields of the CLE by the stoichiometry matrix
$C\in\mathbb{R}^{N_d\times N_r}$ and the propensity vector
$\pmb{v}(\pmb{s})\in\mathbb{R}_{\ge 0}^{N_r}$, the total one‑step MSE can be expressed explicitly as
\begin{equation}
\begin{aligned}
E(\pmb{s};\Delta t,N)&=
\frac14\,\Delta t^2
\sum_{1\le i<j\le N_r}
\sum_{\alpha=1}^{N_s}
\left(
\frac12\sum_{\beta=1}^{N_s}
\Big[
   C_{\beta,i}C_{\alpha,j}
   \sqrt{\tfrac{v_i}{v_j}}\,
   \partial_{s_\beta}v_j
 -
   C_{\beta,j}C_{\alpha,i}
   \sqrt{\tfrac{v_j}{v_i}}\,
   \partial_{s_\beta}v_i
\Big]
\right)^{\!2}
\\[6pt]
&\quad+\,
\frac14\,\Delta t^2
\sum_{i\in \mathcal{I}_{\mathrm{fast}}}\sum_{j\in \mathcal{I}_{\mathrm{slow}}}
\sum_{\alpha=1}^{N_s}
\left(
\frac12\sum_{\beta=1}^{N_s}
\Big[
   C_{\beta,i}C_{\alpha,j}
   \sqrt{\tfrac{v_i}{v_j}}\,
   \partial_{s_\beta}v_j
 -
   C_{\beta,j}C_{\alpha,i}
   \sqrt{\tfrac{v_j}{v_i}}\,
   \partial_{s_\beta}v_i
\Big]
\right)^{\!2}
\\[6pt]
&\quad+\,
\frac{1}{16}\,\frac{\Delta t^2}{N}
\sum_{i,j\in \mathcal{I}_{\mathrm{fast}}} c_{ij}\;
\frac{v_j}{v_i}
\left(
   \sum_{\beta=1}^{N_s} C_{\beta,j}\,\partial_{s_\beta}v_i
\right)^{\!2}
\left(
   \sum_{\alpha=1}^{N_s} C_{\alpha,i}^2
\right)
\\[6pt]
&\quad+\,
\frac{1}{16}\,\Delta t^2
\sum_{j,k\in \mathcal{I}_{\mathrm{slow}}} c_{jk}\;
\frac{v_k}{v_j}
\left(
   \sum_{\beta=1}^{N_s} C_{\beta,k}\,\partial_{s_\beta}v_j
\right)^{\!2}
\left(
   \sum_{\alpha=1}^{N_s} C_{\alpha,j}^2
\right)+\;o(\Delta t^2).
\end{aligned}
\end{equation}

\section{MSE--DRIVEN ADAPTIVE CONTROL FOR FAST--SLOW SPLITTING}

\label{sec:MDE-Driven adpative control}

The error analysis in Section \ref{sec:error} shows that the one‑step mean‑square error (MSE) of the fast–slow splitting integrator admits the asymptotic expansion
$E(s;\Delta t, N)
=
\Delta t^{2}
\left(
A(\pmb{s}) + \frac{B(\pmb{s})}{N}
\right)
+ o(\Delta t^{2})$, which explicitly characterizes the joint influence of the macro‑step size $\Delta t$ and the number $N$ on the method’s accuracy.
Building on this characterization, we develop an adaptive fast–slow splitting strategy with MSE‑driven proportional–integral control (FS–MSE–PI) that jointly selects $\Delta t$ and $N$ to meet a prescribed trajectory‑level error tolerance while minimizing computational cost.

Given a user-specified tolerance $\varepsilon>0$, our objective is to choose
$(\Delta t,N)$ such that
$E(\pmb{s};\Delta t, N) \;\approx\; \varepsilon$. 
Here, $\varepsilon$ represents the one‑step mean‑square deviation of the numerical state from the reference stochastic flow over a macro‑step of length $\Delta t$, thereby quantifying the local error across all species.
Neglecting higher‑order terms for stepsize selection yields the constraint
\(
\Delta t^{2}\!\left( A(\pmb{s}) + \frac{B(\pmb{s})}{N} \right)
\approx \varepsilon,
\)
which motivates an adaptive controller based on real‑time MSE estimates.
Let $E_{n}$ denote the estimated MSE at macro‑step  $n$.  
Since $E$ scales as $\Delta t^{2}$
to leading order, we employ a PI controller of the form
\begin{equation}
\label{eq:PI_controller}
\Delta t_{n+1}
=
\Delta t_n
\left( \frac{\varepsilon}{E_n} \right)^{\alpha}
\left( \frac{E_{n-1}}{E_n} \right)^{\beta},
\end{equation}
with exponents $(\alpha, \beta)$ chosen to stabilize the update in the presence of stochastic
fluctuations. Following standard practice \cite{ilie2015adaptive}, we take
\(
(\alpha, \beta) = (0.2, 0.1)
\) 
and include a safety factor $\theta\in[0.8,0.95]$, 
$
\Delta t_{n+1}
\;\leftarrow\;
\theta\,\Delta t_{n+1}
$
to further enhance robustness.

With $\Delta t_{n+1}$ fixed, the tolerance condition implies
\(
\Delta t_{n+1}^{2}\!\left( A(\pmb{s}) + \frac{B(\pmb{s})}{N} \right)
= \varepsilon.
\) 
If $\varepsilon \le A(\pmb{s})\Delta t_{n+1}^{2}$, the desired tolerance cannot be achieved regardless of $N$ and $\Delta t_{n+1}$ must
be reduced.  Otherwise, selecting
\begin{equation}
\label{eq:N_formula}
N_{n+1}
=
\left\lceil
\frac{B(\pmb{s})\,\Delta t_{n+1}^{2}}
     {\varepsilon - A(\pmb{s})\,\Delta t_{n+1}^{2}}
\right\rceil
\end{equation}
ensures that the fast-substep contribution $B(\pmb{s})\Delta t_{n+1}^{2}/N_{n+1}$
remains below the allowable residual tolerance.
In practice, we enforce bounds 
$1\le N_{n+1}\le N_{\max}$.
If the constraint from~\eqref{eq:N_formula} yields $N_{n+1}>N_{\max}$,
the macro-step $\Delta t_{n+1}$ is reduced and $N_{n+1}$ is recomputed.

To improve robustness in stiff or highly stochastic regimes, we apply standard
safeguards commonly used in adaptive integrators
\cite{ilie2015adaptive}:
(a) a stepsize growth limit
$\Delta t_{n+1}\le r_{\max}\Delta t_n$ with $r_{\max}\in[1.5,2]$;
(b) lower and upper bounds
$\Delta t_{\min}\le\Delta t_{n+1}\le\Delta t_{\max}$;
(c) a minimum substep constraint $N_{n+1}\ge1$ to prevent elimination of the
fast solver; and
(d) a reject--retry rule in which steps with observed error
$E_n>2\varepsilon$ are rejected, $\Delta t_n$ is reduced (typically halved), and
the step is recomputed.
These safeguards are essential for stability and efficiency in long‑time simulations.


\section{EMPIRICAL STUDY}
\label{sec:Empirical Study}

In this section, we evaluate the finite‑sample performance of the proposed adaptive time‑stepping method on representative biochemical reaction networks. The test cases are designed to highlight numerical challenges commonly encountered in stochastic chemical kinetics, including strong stiffness and nonlinear stochastic interactions. As a benchmark, we consider a stiff stochastic reaction network adapted from literature\cite{ilie2015adaptive}, which has been widely used to assess numerical integrators for the chemical Langevin equation in strongly stiff regimes.
The model consists of three species $X_1,X_2,X_3$ interacting through six reversible bimolecular and unimolecular reactions,
\[
\begin{aligned}
X_1+X_2 \;\overset{\kappa_1}{\underset{\kappa_2}{\rightleftarrows}}\; X_3,\qquad
X_1+X_3 \;\overset{\kappa_3}{\underset{\kappa_4}{\rightleftarrows}}\; X_2,\qquad
X_2+X_3 \;\overset{\kappa_5}{\underset{\kappa_6}{\rightleftarrows}}\; X_1,
\end{aligned}
\]
with mass-action propensities
$\lambda_1=\kappa_1 x_1x_2$, $\lambda_2=\kappa_2 x_3$,
$\lambda_3=\kappa_3 x_1x_3$, $\lambda_4=\kappa_4 x_2$,
$\lambda_5=\kappa_5 x_2x_3$, and $\lambda_6=\kappa_6 x_1$,
where $\pmb{x}=(x_1,x_2,x_3)^\top$.

Unless otherwise stated, the reaction rates are fixed as
$\kappa_1=10^2$, $\kappa_2=10^4$, $\kappa_3=10^{-4}$,
$\kappa_4=10^{-2}$, and $\kappa_6=10^3$,
while the coupling parameter $\kappa_5$ is varied to control the stiffness.
Specifically, we consider $\kappa_5\in\{0.1,\,1,\,10\}$.
Although reaction $\mathrm{R}_5$ is classified as slow, it directly couples
species participating in the fast reactions, thereby strengthening the interaction
between slow and fast subsystems as $\kappa_5$ increases.

\begin{table}[htpb]
\centering
\small
\caption{\small Distributional errors for different methods under varying stiffness parameter $\kappa_5$.
We measure distributional accuracy using the relative Wasserstein--1 distance
$\mathrm{Rel.\ }W_1=W_1(P,Q)/\mu_{\mathrm{ref}}$, where
$W_1(P,Q)=\int |F_P-F_Q|\,dx$ and
$\mu_{\mathrm{ref}}=\mathbb{E}[X^{\mathrm{ref}}(T)]$.
Moment accuracy is quantified by the relative mean error
$|\mathbb{E}[X^{\mathrm{meth}}(T)]-\mathbb{E}[X^{\mathrm{ref}}(T)]|/\mu_{\mathrm{ref}}$
and the relative variance error
$|\mathrm{Var}(X^{\mathrm{meth}}(T))-\mathrm{Var}(X^{\mathrm{ref}}(T))|
/\mathrm{Var}(X^{\mathrm{ref}}(T))$.
We further report the Jensen--Shannon divergence and the
Kullback--Leibler divergence computed from Gaussian kernel density estimates. All reported values are shown as mean $\pm$ 95\% confidence intervals,
computed from 10 independent repetitions, each based on $10{,}000$
sample paths.}
\label{tab:distribution_errors}
\renewcommand{\arraystretch}{1.15}
\begin{subtable}{\textwidth}
\centering
\caption{$\kappa_5 = 0.01$ (Steps = 1373)}
\begin{tabular}{clccccc}
\toprule
Species & Method & Rel. W$_1$ & Rel. Mean Error & Rel. Var. Error & JS Div. \\
\midrule
\multirow{3}{*}{$X_0$}
& FS–MSE–PI
& \textbf{2.186} $\pm$ 0.088
& \textbf{1.095} $\pm$ 0.222
& \textbf{0.144} $\pm$ 0.012
& \textbf{0.002} $\pm$ 0.000\\
& EM
& 37.248 $\pm$ 4.213
& 31.969 $\pm$ 0.325
& 4.046 $\pm$ 1.665
& 0.171 $\pm$ 0.026\\
& Ilie--PI
& 16.842 $\pm$ 0.185
& 1.848 $\pm$ 0.509
& 2.822 $\pm$ 0.073
& 0.078 $\pm$ 0.001\\
\midrule
\multirow{3}{*}{$X_1$}
& FS–MSE–PI
& 0.121 $\pm$ 0.014
& 0.108 $\pm$ 0.016
& \textbf{0.019} $\pm$ 0.007
& 0.001 $\pm$ 0.000\\
& EM
& 2.299 $\pm$ 0.014
& 2.299 $\pm$ 0.014
& 0.111 $\pm$ 0.031
& 0.089 $\pm$ 0.002\\
& Ilie--PI
& \textbf{0.079} $\pm$ 0.012
& \textbf{0.068} $\pm$ 0.018
& 0.043 $\pm$ 0.013
& \textbf{0.001} $\pm$ 0.000\\
\midrule
\multirow{3}{*}{$X_2$}
& FS–MSE–PI
& \textbf{0.198} $\pm$ 0.008
& \textbf{0.102} $\pm$ 0.021
& \textbf{0.142} $\pm$ 0.012
& \textbf{0.002} $\pm$ 0.000\\
& EM
& 29.286 $\pm$ 0.090
& 29.286 $\pm$ 0.090
& 31.922 $\pm$ 1.596
& 0.665 $\pm$ 0.004\\
& Ilie--PI
& 1.503 $\pm$ 0.018
& 0.167 $\pm$ 0.046
& 2.730 $\pm$ 0.072
& 0.075 $\pm$ 0.001\\

\bottomrule
\end{tabular}
\end{subtable}

\vspace{0em}

\begin{subtable}{\textwidth}
\centering
\caption{$\kappa_5 = 0.1$ (Steps = 1207)}
\begin{tabular}{clccccc}
\toprule
Species & Method & Rel. W$_1$ & Rel. Mean Error & Rel. Var. Error & JS Div.\\
\midrule
\multirow{3}{*}{$X_0$}
& FS–MSE–PI
& \textbf{1.587} $\pm$ 0.067
& 0.692 $\pm$ 0.140
& \textbf{0.111} $\pm$ 0.009
& \textbf{0.001} $\pm$ 0.000\\
& EM
& 16.983 $\pm$ 2.401
& 9.917 $\pm$ 0.318
& 2.460 $\pm$ 0.585
& 0.090 $\pm$ 0.019\\
& Ilie--PI
& 10.831 $\pm$ 0.127
& \textbf{0.562} $\pm$ 0.203
& 1.867 $\pm$ 0.104
& 0.053 $\pm$ 0.001\\
\midrule
\multirow{3}{*}{$X_1$}
& FS–MSE–PI
& \textbf{0.057} $\pm$ 0.005
& \textbf{0.014} $\pm$ 0.010
& \textbf{0.011} $\pm$ 0.008
& \textbf{0.000} $\pm$ 0.000\\
& EM
& 1.110 $\pm$ 0.016
& 1.110 $\pm$ 0.016
& 0.055 $\pm$ 0.018
& 0.015 $\pm$ 0.000\\
& Ilie--PI
& 0.128 $\pm$ 0.013
& 0.097 $\pm$ 0.023
& 0.057 $\pm$ 0.011
& 0.001 $\pm$ 0.000\\
\midrule
\multirow{3}{*}{$X_2$}
& FS–MSE–PI
& \textbf{0.188} $\pm$ 0.008
& 0.086 $\pm$ 0.015
& \textbf{0.108} $\pm$ 0.009
& \textbf{0.001} $\pm$ 0.000\\
& EM
& 8.876 $\pm$ 0.051
& 8.864 $\pm$ 0.055
& 7.880 $\pm$ 0.590
& 0.331 $\pm$ 0.006\\
& Ilie--PI
& 1.273 $\pm$ 0.015
& \textbf{0.067} $\pm$ 0.025
& 1.828 $\pm$ 0.101
& 0.052 $\pm$ 0.001\\

\bottomrule
\end{tabular}
\end{subtable}

\vspace{0em}

\begin{subtable}{\textwidth}
\centering
\caption{$\kappa_5 = 1$ (Steps = 860)}
\begin{tabular}{clccccc}
\toprule
Species & Method & Rel. W$_1$ & Rel. Mean Error & Rel. Var. Error & JS Div.\\
\midrule
\multirow{3}{*}{$X_0$}
& FS–MSE–PI
& \textbf{1.248} $\pm$ 0.083
& 1.084 $\pm$ 0.110
& \textbf{0.134} $\pm$ 0.014
& \textbf{0.002} $\pm$ 0.000\\
& EM
& 89.151 $\pm$ 0.156
& 89.151 $\pm$ 0.156
& 3.990 $\pm$ 0.288
& 0.655 $\pm$ 0.005 \\
& Ilie--PI
& 3.147 $\pm$ 0.056
& \textbf{0.369} $\pm$ 0.154
& 0.713 $\pm$ 0.016
& 0.017 $\pm$ 0.001\\
\midrule
\multirow{3}{*}{$X_1$}
& FS–MSE–PI
& 0.536 $\pm$ 0.037
& 0.534 $\pm$ 0.037
& \textbf{0.017} $\pm$ 0.011
& \textbf{0.001} $\pm$ 0.000\\
& EM
& 0.408 $\pm$ 0.081
& 0.057 $\pm$ 0.025
& 0.186 $\pm$ 0.037
& 0.002 $\pm$ 0.001\\
& Ilie--PI
& \textbf{0.199} $\pm$ 0.019
& \textbf{0.040} $\pm$ 0.024
& 0.083 $\pm$ 0.010
& \textbf{0.001} $\pm$ 0.000\\
\midrule
\multirow{3}{*}{$X_2$}
& FS–MSE–PI
& \textbf{0.392} $\pm$ 0.026
& 0.339 $\pm$ 0.035
& \textbf{0.135} $\pm$ 0.014
& \textbf{0.002} $\pm$ 0.000\\
& EM
& 89.817 $\pm$ 0.109
& 89.817 $\pm$ 0.109
& 17.331 $\pm$ 0.246
& 0.693 $\pm$ 0.000\\
& Ilie--PI
& 0.996 $\pm$ 0.018
& \textbf{0.121} $\pm$ 0.051
& 0.711 $\pm$ 0.016
& 0.017 $\pm$ 0.001\\
\bottomrule
\end{tabular}
\end{subtable}

\end{table}

For the cases with $\kappa_5=0.01,0.1$, the initial condition is set to
$\pmb{x}(0)=(150,100,50)^\top$, whereas for $\kappa_5=1$ we use
$\pmb{x}(0)=(100,1000,100)^\top$.
These initial conditions are chosen to place the system in representative
operating regimes and to avoid transient extinction or saturation effects
across the different parameter settings.
All simulations are performed on the time interval $[0,\,2\times10^{-2}]$, following the experimental setup in the literature study \cite{ilie2015adaptive}.

The stiffness of the system arises from the large disparity among the reaction rate parameters, which induces a pronounced separation of time scales.
Reactions with large rate constants dominate the fast dynamics, while the remaining reactions evolve more slowly. Accordingly, we classify the reaction channels as follows: (1) \textbf{Fast reactions:}
$\mathrm{R}_1$ ($\kappa_1=10^2$), $\mathrm{R}_2$ ($\kappa_2=10^4$), and $\mathrm{R}_6$ ($\kappa_6=10^3$); and (2) \textbf{Slow reactions:}
$\mathrm{R}_3$ ($\kappa_3=10^{-4}$), $\mathrm{R}_4$ ($\kappa_4=10^{-2}$), and $\mathrm{R}_5$ (with $\kappa_5=0.1,\,1,\,10$).

\begin{figure}[t]
\centering
\includegraphics[width=0.92\textwidth]{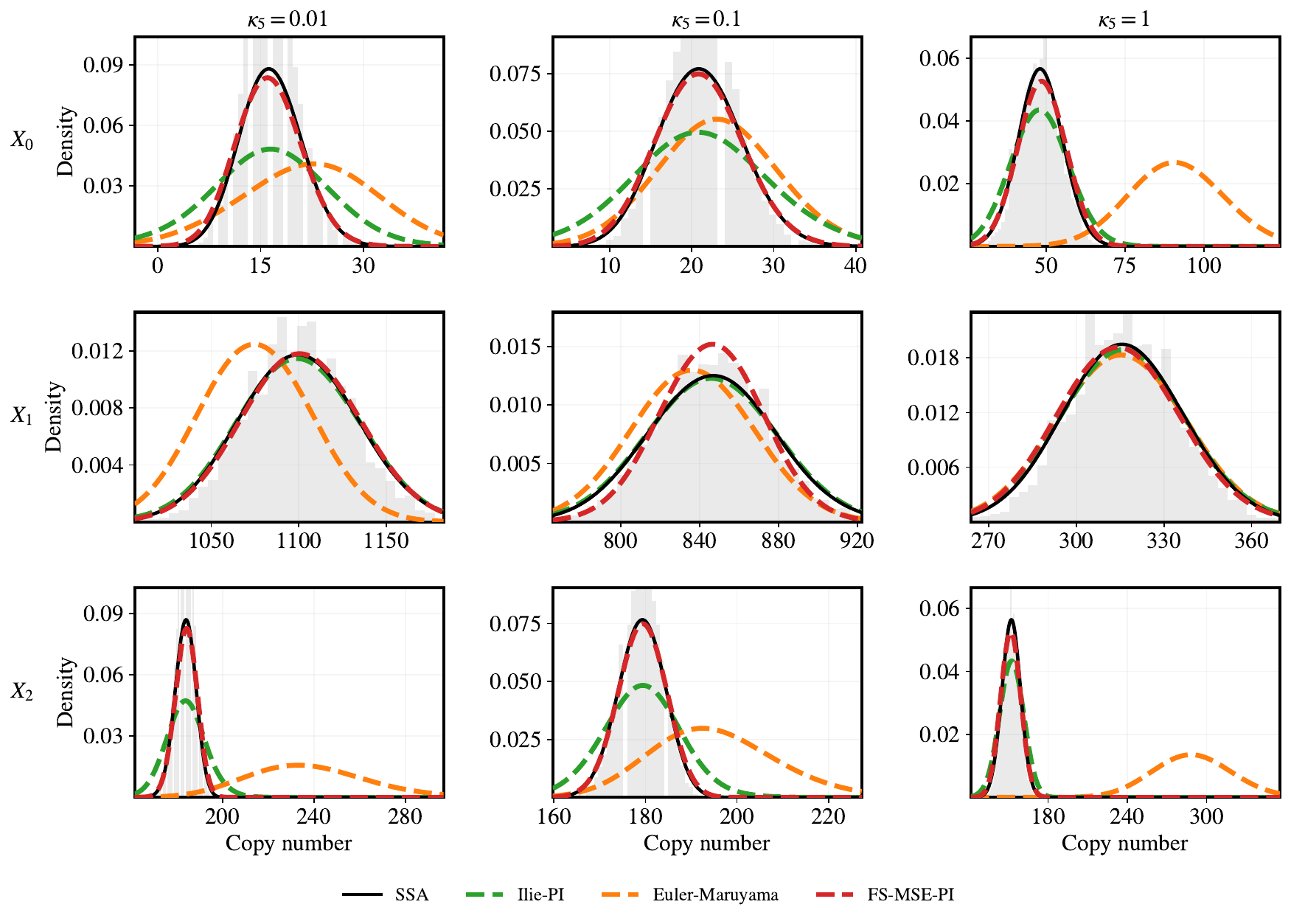}
\vspace{-0.5em} 
\caption{
Empirical marginal distributions for species $X_0$--$X_2$ (rows) under increasing stiffness parameter $\kappa_5$ (columns).
In each panel, the solid black curve denotes the SSA reference density, dashed colored curves correspond to fixed-step splitting, Euler--Maruyama, and the adaptive method, and the gray shaded bars indicate the empirical histogram of SSA samples.
}
\label{fig:distribution_comparison_matrix}
\end{figure}

While $\kappa_5$ nominally scales a slow reaction, it modulates exchanges among the same species involved in the fast reaction loops. Consequently, varying
$\kappa_5$ alters not only the nominal time‑scale separation but also the interaction structure between fast and slow channels. In particular, changes in $\kappa_5$ affect the local vector fields and their derivatives, thereby modifying the magnitude of fast–slow commutator terms that govern splitting errors. For this reason, we treat $\kappa_5$ as a practical control parameter for probing different fast–slow coupling regimes in this benchmark, while recognizing that commutator norms (or related quantitative measures) provide a more precise characterization of coupling strength.

This example is used to evaluate the effectiveness of the proposed
adaptive fast--slow splitting method with MSE--based PI control (FS--MSE--PI)
in handling stiff stochastic dynamics.
We compare FS--MSE--PI with two reference approaches:
a fixed-step Euler--Maruyama discretization of the chemical Langevin equation,
and the PI‑adaptive CLE method of Ilie and Morshed \cite{ilie2015adaptive} (denoted Ilie–PI). For a fair comparison, all methods are executed using the same number of macro time steps.

To quantitatively assess the accuracy of the proposed method, we compare the empirical distributions obtained from each numerical scheme with those generated by Gillespie’s exact stochastic simulation algorithm (SSA), which serves as a reference solution for the chemical master equation.
Let $X^{\mathrm{ref}}(T)$ and $X^{\mathrm{meth}}(T)$ denote the random variables corresponding to the reference SSA solution and a given numerical method at the final time $T$, respectively.
All evaluation metrics reported below are computed from Monte Carlo samples of these terminal‑time distributions.

Table~\ref{tab:distribution_errors} compares distributional and moment
errors for three species \(X_0,X_1,X_2\) under different values of
\(\kappa_5\), for Euler--Maruyama (EM), the PI-adaptive CLE method of Ilie
and Morshed~\cite{ilie2015adaptive} (Ilie--PI), and the proposed adaptive
fast--slow splitting method with MSE-based PI control (FS--MSE--PI). FS--MSE--PI yields substantially smaller Wasserstein, Jensen--Shannon, and KL
errors for the fast-sensitive species \(X_0\) and \(X_2\) across the tested
parameter regimes; improvements are most pronounced at larger \(\kappa_5\).
For the slow-dominated species \(X_1\), FS--MSE--PI and Ilie--PI give
comparable accuracy. Ilie--PI uses a broadly comparable macro-step budget to
FS--MSE--PI, 
while EM is run on a fixed-step grid.
These results indicate that jointly adapting macro steps and fast substepping
(as in FS--MSE--PI) better controls fast–slow interaction errors that dominate distributional accuracy for stiffness-sensitive species.

Figure~\ref{fig:distribution_comparison_matrix} compares the terminal distributions
produced by the different methods with the SSA reference for all species and
values of $\kappa_5$. Consistent with Table~\ref{tab:distribution_errors},
FS--MSE--PI closely matches the SSA distributions for the fast--sensitive species
$X_0$ and $X_2$ across all regimes, accurately capturing both mean and spread.
Euler--Maruyama shows pronounced distortions, including biased means and
over--dispersed right tails, especially for $X_0$ and $X_2$, while Ilie--PI
improves over EM but still exhibits visible broadening or skewness.


We compare the wall-clock cost of FS--MSE--PI and Ilie--PI across different
regimes. Taking FS--MSE--PI as the baseline (100\%), Ilie--PI requires
approximately $163\%$, $173\%$, and $182\%$ of the computational cost for
$\kappa_5 = 0.01$, $0.1$, and $1$, respectively. Thus, FS--MSE--PI is consistently more efficient despite the overhead of
fast--slow splitting and error estimation. Moreover, the relative cost of
Ilie--PI increases with $\kappa_5$, from about $1.6\times$ to $1.8\times$
that of FS--MSE--PI, reflecting the growing difficulty of controlling
distributional errors via time-step adaptivity alone. By jointly adapting
the macro step size and fast substepping resolution, FS--MSE--PI allocates
computational effort more effectively across regimes.

Taken together with the distributional accuracy results in
Table~\ref{tab:distribution_errors} and Figure~\ref{fig:distribution_comparison_matrix},
these timings demonstrate that FS--MSE--PI achieves improved distributional
accuracy for fast-sensitive species at a lower overall computational cost than
time-step adaptivity alone.

\section{CONCLUSION}
\label{sec:conclusion}



In this paper, we developed an error-driven framework for fast--slow operator
splitting of the Chemical Langevin Equation based on a stochastic-flow
representation. This formulation enables a systematic Lie‑bracket characterization of splitting errors and leads to a complete mean‑square error (MSE) decomposition. We show that the leading‑order error comprises stochastic‑flow truncation, fast–slow commutator effects, and discretization errors, with explicit dependence on the macro‑step size
$\Delta t$ and the number of fast substeps $N$. Leveraging this structure, we proposed an adaptive strategy that jointly selects
$\Delta t$ and $N$ to achieve a prescribed accuracy while reducing computational cost. Numerical experiments demonstrate that the proposed method consistently improves both accuracy and efficiency relative to existing approaches.

\section*{ACKNOWLEDGEMENT}
We gratefully acknowledge funding support from the National Science Foundation (Grant CAREER CMMI-2442970) and the National Institute of Standards and Technology (Grant nos. 70NANB24H293, 70NANB21H086, 70NANB17H002).


\footnotesize

\bibliographystyle{unsrt}

\bibliography{wscref,proposal,proj_ref}

\end{document}